\documentclass[preprint]{article}

\usepackage{amsmath,amssymb,amsthm,
verbatim,
}
\usepackage[utf8]{inputenc}

\usepackage{tikz}
\usetikzlibrary{shapes.geometric,calc,decorations.markings,math}

\tikzstyle{nodo}=[circle,draw,fill,inner sep=0pt,minimum size=%
1.5mm]
\tikzstyle{infinito}=[circle,inner sep=0pt,minimum size=0mm]

\usepackage{geometry}
\newcommand\R{{\mathbb R}}

\newcommand\Hmu{{H_\mu^1}}

\newcommand\f{\frac}

\newcommand{\sgn}{\text{sgn}}

\newtheorem{theorem}{Theorem}[section]
\newtheorem{proposition}[theorem]{Proposition}

\theoremstyle{remark}
\newtheorem{remark}[theorem]{Remark}
\newtheorem*{remark*}{Remark}

\theoremstyle{definition}

\theoremstyle{plain}
\newtheorem{thm}{Theorem}[section]

\newtheorem{lem}[thm]{Lemma}
\newtheorem{prop}[thm]{Proposition}

\theoremstyle{definition}

\theoremstyle{definition}

\theoremstyle{remark}

\tikzset{every loop/.style={min distance=10mm,in=300,out=240,looseness=10}}

\tikzset{place/.style={circle,thick,draw=blue!75,fill=blue!20,minimum
size=6mm}}
\tikzset{place2/.style={circle,thick,draw=red!75,fill=red!20,minimum
size=6mm}}

\newcommand{\deb}{\rightharpoonup}
\newcommand{\F}{\mathcal{F}}

\newcommand{\tbar}{\overline{t}}

\newcommand{\udot}{\|u'\|_2}
\newcommand{\uLtwo}{\|u\|_2}
\newcommand{\uinfty}{\|u\|_\infty}
\newcommand{\uLp}{\|u\|_p}
\newcommand{\uLsix}{\|u\|_6}
\newcommand{\DD}{\mathcal{D}}

\title{Prescribed mass ground states for a doubly nonlinear Schr\"odinger equation in dimension one}

\author{Filippo Boni$^{\dagger,\sharp}$\thanks{The first author acknowledges that the present research has been partially supported by MIUR grant Dipartimenti di Eccellenza 2018-2022 (E11G18000350001)}\, and  Simone Dovetta$^\ddagger$ \\
	\\
	$^\dagger$Dipartimento di Scienze Matematiche ``G.L. Lagrange'' \\
	Politecnico di Torino \\
	C.so Duca Degli Abruzzi 24, 10129 Torino, Italy \\ 
	\\ \ \\
	$^\sharp$Dipartimento di Matematica ``G. Peano''\\
	Università degli Studi di Torino 
\\	Via Carlo Alberto, 10, 10123, Torino, Italy\\
	\\
	filippo.boni@polito.it\\ \ \\
\\
$^\ddagger$Istituto di Matematica Applicata e Tecnologie Informatiche "E. Magenes"\\
Via Adolfo Ferrata, 1, 27100, Pavia, Italy\\
\\
simone.dovetta@imati.cnr.it
}

\date{\today}

\begin{document}


\maketitle

\begin{abstract}
	
	We investigate the problem of existence and uniqueness of ground states at fixed mass for two families of focusing nonlinear Schr\"odinger equations on the line. 
	
	The first family consists of NLS with power nonlinearities concentrated at a point. For such model, we prove existence and uniqueness of ground states at every mass when the nonlinearity power is $L^2-$subcritical and at a threshold value of the mass in the $L^2-$critical regime. 
	
	The second family is obtained by adding a standard power nonlinearity to the previous setting. In this case, we prove existence and uniqueness at every mass in the doubly subcritical case, namely when both the powers related to the pointwise and the standard nonlinearity are subcritical. If only one power is critical, then existence and uniqueness hold only at masses lower than the critical mass associated to the critical nonlinearity. Finally, in the doubly critical case ground states exist only at critical mass, whose value results from a non--trivial interplay between the two nonlinearities.
	
\end{abstract}

\section{Introduction}

In this paper we investigate a class of doubly nonlinear Schr\"odinger equations on the real line of the form
\begin{equation}
	\label{eq Delta}
	i\partial_t \psi(x,t)=-\partial_{xx}^2 \psi(x,t)-|\psi(x,t)|^{p-2}\psi(x,t)-|\psi(x,t)|^{q-2}\delta_0 \psi(x,t)
\end{equation}
with a focusing standard nonlinearity and a pointwise focusing nonlinearity located at the origin.

Through the decades, the use of the nonlinear Schr\"odinger equation as a mathematical tool to describe reality has rapidly become prominent and widespread in several areas of physics, from the theory of Bose--Einstein condensates \cite{DPS} to the propagation of laser beams \cite{GSD, PP}, from signal transmission in a neuronal network \cite{CM} to fluid dynamics \cite{L}. 
The interest in nonlinear Schr\"odinger equations with nonlinearity confined in a localized region of the space dates back to the early Nineties, mainly driven by the physical analysis of the dynamics of a quantum particle running through a barrier or some impurity in a medium (see for instance \cite{BKB,CH1, CH2, CH3,JPS,MA,MB,N,PJC,SKB,SKBRC} and references therein, as well as the monograph \cite{AGHKH}). Since then, well--posedness and global solutions have been studied both on the real line $\R$ \cite{AT,AT JFA} and in dimension three \cite{ADFT1,ADFT2} first, whereas recent investigations have been devoted to the two--dimensional case \cite{ACCT,ACCT2,CCT,CFT}. A rigorous derivation of the model with pointwise nonlinearity from the standard NLS equation can be found in \cite{CFNT1,CFNT2}. In the one--dimensional setting, a detailed blow--up analysis has been developed in \cite{HL1,HL2} for the model with concentrated nonlinearities, and the discussion of the interaction between a standard nonlinearity and a linear delta has been started in \cite{BV}, dealing with scattering issues. More recently, similar settings have been considered also in the case of quantum beating \cite{CFN17}, fractional Schr\"odinger equations \cite{CFiT} and on non--compact metric graphs (see \cite{DT-p,DT,serratentarelli2016,ST-NA,T-JMAA} and \cite{BCT1,BCT,BCT2,CCNP} for the nonlinear Dirac equation).

The present paper fits in this line of research. Specifically, we address existence and uniqueness of ground states of the energy functional $F_{p,q}:H^1(\R)\to\R$ associated to \eqref{eq Delta}
\begin{equation}
\label{F_intro}
F_{p,q}(u):=\frac{1}{2}\|u'\|_{L^2(\R)}^2-\frac{1}{p}\|u\|_{L^p(\R)}^p-\frac{1}{q}|u(0)|^q\,,
\end{equation}
i.e. global minimizers of $F_{p,q}$ among all functions $u\in H^1(\R)$ fulfilling the mass constraint
\[
\|u\|_{L^2(\R)}^2=\mu>0\,.
\]
In other words, we seek solutions of the problem
\begin{equation}
\label{infF_intro}
\F_{p,q}(\mu):=\inf_{u\in\Hmu(\R)}F_{p,q}(u)
\end{equation}
with
\[
\Hmu(\R):=\{\,u\in H^1(\R)\,:\,\|u\|_{L^2(\R)}^2=\mu\,\}\,.
\]
Ground states are solutions of the eigenvalue problem associated to \eqref{eq Delta}
\begin{equation}
\label{stat eq delta}
u''+|u|^{p-2}u+|u|^{q-2}\delta_0 u=\omega u\qquad\text{on }\R\,.
\end{equation}
Moreover, according to the standard theory of stability \cite{CL,GSS}, given any ground state $u$, it is well--known that the function $\psi(x,t)=e^{i\omega t}u(x)$ is an orbitally stable standing wave of \eqref{eq Delta}.

We begin by considering a simplified model, looking for ground states of the functional $D_q:H^1(\R)\to\R$
\begin{equation}
\label{D_intro}
D_q(u):=\f12\|u'\|_{L^2(\R)}^2-\f1q|u(0)|^q
\end{equation}
with a pointwise nonlinearity located at the origin only, that is we consider the minimization problem
\begin{equation}
\label{infD_intro}
\DD_q(\mu):=\inf_{u\in\Hmu(\R)}D_q(u)\,.
\end{equation}
The corresponding time--dependent NLS equation is then
\[
i\partial_t \psi(x,t)=-\partial_{xx}^2 \psi(x,t)-|\psi(x,t)|^{q-2}\delta_0 \psi(x,t)
\]
whose associated eigenvalue problem reads
\begin{equation}
	\label{Lstat delta}
	u''+|u|^{q-2}\delta_0 u=\omega u\qquad\text{on }\R\,.
\end{equation}
The first theorem identifies a subcritical regime $q\in(2,4)$, where ground states of $D_q$ exist and are unique at every value of the mass.
\begin{thm}
	\label{THM D 1} Let $2<q<4$. Then, for every $\mu>0$, 
	\begin{equation}
	\label{eq_THM D 1}
	-\infty<\DD_q(\mu)<0
	\end{equation}
	and there always exists a unique positive ground state $\chi_q\in\Hmu(\R)$ at mass $\mu$ given by
	\[
	\chi_q(x)=\left(\f\mu2\right)^{\f{q-2}{4-q}}\exp\left(-2^{-\f2{q-2}}\mu^{\f{q-2}{4-q}}|x|\right)\,,\qquad x\in\R\,.
	\]
\end{thm}

The situation changes at $q=4$, as depicted in the following theorem.

\begin{thm}
	\label{THM D 2} Let $q=4$. Then ground states at mass $\mu$ exist if and only if $\mu=2$ and
	\begin{equation}
	\label{eq_THM D 2}
	\DD_4(\mu)=\begin{cases}
	0 & \text{if }\mu\leq2\\
	-\infty & \text{if }\mu>2
	\end{cases}
	\end{equation}
	Moreover, there exists a family $\{\chi_\lambda\}_{\lambda>0}\in H_2^1(\R)$ of positive ground states at mass $2$, given by
	\[
	\chi_\lambda(x):=\sqrt{2\lambda}e^{-\lambda|x|}\,,\qquad x\in\R\,.
	\]
\end{thm}

Notice that, when $q=4$, the threshold value of the mass $\mu=2$ appears and the infimum of the energy undertakes a sharp transition from 0 to $-\infty$ crossing this critical mass. Moreover, ground states exist only at the threshold. For this reason, we call $q=4$ the $L^2-$critical nonlinearity of the model. 

A similar behaviour is well--known (see \cite{cazenave}) in the case of the standard nonlinearity only, i.e. for the energy functional
\begin{equation*}
\label{E_intro}
E_p(u):=\f12\|u'\|_{L^2(\R)}^2-\f1p\|u\|_{L^p(\R)}^p\,,
\end{equation*}
for which the $L^2-$critical power is $p=6$. Indeed, in the $L^2-$subcritical setting $p\in(2,6)$, it has been shown that, for every $\mu>0$ there exists a unique (up to translations) ground state at mass $\mu$, the so--called soliton.
Conversely, at the critical power $p=6$, the infimum of the energy passes from 0 to $-\infty$ as the mass crosses the critical value $ \mu=\f{\sqrt{3}}2\pi$
and a family of solitons exists only at this threshold value.

Turning to the doubly nonlinear functional $F_{p,q}$, as one may expect, existence and uniqueness of ground states hold for every value of $\mu$ whenever $p\in(2,6)$ and $q\in(2,4)$, namely when both nonlinearities are subcritical. This is the content of the following theorem.
\begin{thm}[Doubly subcritical regime]
	\label{THM F 1}
	Let $2<p<6$ and $2<q<4$. Then, for every $\mu>0$, 
	\begin{equation}
	\label{eq_THM F 1}
	-\infty<\F_{p,q}(\mu)<0
	\end{equation}
	and there always exists a unique positive ground state at mass $\mu$.
\end{thm}

Apparently, up to now no significant interaction between the two nonlinear terms takes place. This is no longer true when we take into account the critical powers. The interplay between a critical and a subcritical nonlinearity is unravelled in the next result.

\begin{thm}[Single critical regime]
	\label{THM F 2}
	Let $\mu>0$. 
	\begin{itemize}
		\item[(i)] If $p=6$ and $2<q<4$, then there exists a unique positive ground state at mass $\mu$ if and only if $\mu<\f{\sqrt{3}}2 \pi$, and
		\begin{equation}
		\label{eq_THM F 2}
		\begin{cases}
		-\infty<\F_{6,q}(\mu)<0 & \text{if }\mu<\f{\sqrt{3}}2 \pi\\
		\F_{6,q}(\mu)=-\infty & \text{if }\mu\geq\f{\sqrt{3}}2 \pi\,.
		\end{cases}
		\end{equation}
		\item[(ii)] If $2<p<6$ and $q=4$, then there exists a unique positive ground state at mass $\mu$ if and only if $\mu<2$
		\begin{equation}
		\label{eq_THM F 3}
		\begin{cases}
		-\infty<\F_{p,4}(\mu)<0 & \text{if }\mu<2\\
		\F_{p,4}(\mu)=-\infty & \text{if }\mu\geq2\,.
		\end{cases}
		\end{equation}
	\end{itemize}
\end{thm}
First, notice that it is enough to consider one nonlinear term at its critical power to ensure the appearance of threshold phenomena. Moreover, the critical value of the mass is insensitive to the subcritical nonlinearity, as it corresponds to the threshold of the problem with the critical term only. On the other hand, the effect of the subcritical nonlinearity is evident in the range of masses smaller than the threshold, ensuring existence and uniqueness of ground states for all these values of $\mu$. Furthermore, the absence of ground states at the critical mass marks a difference with respect to the purely critical cases and suggests that the passage from boundedness to unboundedness from below
is smoothened by the presence of the subcritical power.

The last result concerns the doubly critical case, where simultaneously $p=6$ and $q=4$. Here we recover the typical structure of a purely critical setting, with the ground state energy level lifting from 0 to $-\infty$ when exceeding a critical value of the mass and solutions existing only at the threshold. A quite remarkable feature due to the interaction between the two nonlinearities is given by the fact that the critical mass \eqref{mustar} is lower than the critical masses $\f{\sqrt{3}}2$ and $2$ for the standard and pointwise nonlinearity.
\begin{thm}[Doubly critical regime]
	\label{THM F 4}
	Let $p=6$ and $q=4$. Then the functional \eqref{F_intro} admits ground states only at mass 
	\begin{equation}
	\label{mustar}
	\mu^*:=\sqrt{3}\left(\frac{\pi}{2}-\arcsin\left(\sqrt{\frac{3}{7}}\right)\right)
	\end{equation}
	and
	\begin{equation}
	\label{eq_THM F 4}
	\F_{6,4}(\mu)=\begin{cases}
	0 & \text{if }\mu\leq\mu^*\\
	-\infty & \text{if }\mu>\mu^*\,.
	\end{cases}
	\end{equation}
\end{thm}
Note that equation \eqref{stat eq delta} corresponds to the stationary NLS equation
\begin{equation}
\label{NLSE}
u''+|u|^{p-2}u=\omega u\qquad\forall x\neq0
\end{equation} 
coupled with the nonlinear condition at the origin
\begin{equation}
\label{matching}
u'(0^-)-u'(0^+)=u(0)|u(0)|^{q-2}\,.
\end{equation}
Since the only positive $L^2-$solution of equation \eqref{NLSE} on $\R$ is the soliton 
\begin{equation}
\label{soliton}
\phi(x)=\left[\f p2\omega\left(1-\tanh^2\left(\f{p-2}2\sqrt{\omega}(|x|)\right)\right)\right]^{\f 1{p-2}}
\end{equation}
possibly translated, the ground states of \eqref{F_intro} can be constructed by pasting two pieces of soliton together, so that the matching condition \eqref{matching} is satisfied. In this way, one obtains
\[
u(x)=\left[\f p2\omega\left(1-\tanh^2\left(\f{p-2}2\sqrt{\omega}(|x|+a)\right)\right)\right]^{\f 1{p-2}}
\]
with $a$ given by the equation
\[
\f{2\tanh\left(\f{p-2}2\sqrt{\omega}a\right)}{\left[1-\tanh^2\left(\f{p-2}2\sqrt{\omega}a\right)\right]^{\f{q-2}{p-2}}}=\left(\f p2\right)^{\f{q-2}{p-2}}\omega^{\f{q-2}{p-2}-\f12}\,.
\]

\medskip
The remainder of the paper is organised as follows. In Section \ref{sec:prelim} we recall some preliminaries and discuss a general compactness argument. Section \ref{sec:D} deals with the purely pointwise nonlinear functional $D_q$, proving Theorems \ref{THM D 1}--\ref{THM D 2}, whereas Sections \ref{sec:subcrit}--\ref{sec:single_crit}--\ref{sec:F64} develop the analysis of the functional $F_{p,q}$. Specifically, Section \ref{sec:subcrit} treats the subcritical regime, exhibiting the proof of Theorem \ref{THM F 1}, while the discussion of the critical cases is given in Section \ref{sec:single_crit} for a single critical exponent (Theorem \ref{THM F 2}) and in Section \ref{sec:F64} for the case of both nonlinearities at the critical power (Theorem \ref{THM F 4}).

\medskip\noindent\textit{Notation.}	In what follows, we use symbols like $\uLp$ to denote $\|u\|_{L^p(\R)}$.

\section{Preliminaries and compactness}
\label{sec:prelim}

We begin here by revising some useful tools that will be helpful in the subsequent analysis. Particularly, we show that, both for  $D_q$ and $F_{p,q}$, if the infimum of the energy in $\Hmu(\R)$ is finite and strictly negative, then ground states at mass $\mu$ exist.

Before doing this, let us recall the well--known Gagliardo--Nirenberg inequality
\begin{equation}
\label{GNp}
\uLp^p\leq K_p \uLtwo^{\frac{p}{2}+1}\udot^{\frac{p}{2}-1}\,,\quad u\in H^1(\R),\,\,p\geq2,\,
\end{equation}
with $K_p>0$ the smallest constant for which the inequality is satisfied. Particularly, when $p=6$, \eqref{GNp} reads
\begin{equation}
	\label{GN6}
	\uLsix^6\leq K_6\uLtwo^4\udot^2
\end{equation}
with $K_6=\f4{\pi^2}$ (see \cite{dolbeaut}). Furthermore, equality in \eqref{GNp}--\eqref{GN6} is attained if and only if (up to translations, dilations and phase) $u$ is the soliton $\phi_\omega$ as in \eqref{soliton}.

When $p=+\infty$, the following version of the inequality  holds
\begin{equation}
	\label{GNinf}
	\uinfty^2\leq \uLtwo\udot\,,\quad u\in H^1(\R)\,.
\end{equation}
Moreover, equality in \eqref{GNinf} is realized if and only if (up to translations, dilations and phase) $u(x)=e^{-|x|}$, $x\in\R$.

\begin{remark}
	\label{rem_coercivity}
	Note that, for every $\mu>0$ and $2<q<4$, the coercivity of $D_q$ in $\Hmu(\R)$ is granted by inequality \eqref{GNinf}. Similarly, \eqref{GNp} and \eqref{GNinf} ensure that $F_{p,q}$ is coercive in $\Hmu(\R)$ for every $\mu>0$ and every $2<p<6$ and $2<q<4$.
\end{remark}

\begin{remark}
	\label{rem_symmetry}	
	By a standard rearrangement argument, it is readily seen that the minimization problems \eqref{infD_intro} and \eqref{infF_intro} can be restricted to the subspace of real, non--negative functions $u\in\Hmu(\R)$ that are even and non--increasing on $(0,+\infty)$. 
	Indeed, up to replacing $u$ with $|u|$, we can assume $u$ real and $u\geq0$. 
	Moreover, given a real $u\in\Hmu(\R)$, $u\geq0$, and denoted by $\widehat{u}\in\Hmu(\R)$ its symmetric rearrangement 
	\[
	\widehat{u}(x):=\inf\{\,t\geq0\,:\,|\{u>t\}|\leq2|x|\,\},\qquad x\in\R,
	\]
	we have
	\[
	\|\widehat{u}'\|_2\leq \udot,\quad \|\widehat{u}\|_p=\uLp\quad \text{and}\quad |\widehat{u}(0)|\geq|u(0)|\,,
	\]
	entailing
	\[
	D_q(\widehat{u})\leq D_q(u)\quad\text{and}\quad F_{p,q}(\widehat{u})\leq F_{p,q}(u)\,.
	\]
\end{remark}
We are now ready to discuss the following compactness criterion. Such a result is rather natural, and it exploits in our setting the well--known compactness of the embedding in $L^r$ of radial $H^1$ functions (see for instance \cite[Appendix A.II]{BL} and \cite[Section 1.7]{cazenave}).

\begin{proposition}
	\label{prop_compactness}
	Assume $J=D_q$ with $2<q<4$, or $J=F_{p,q}$ with $2<p<6,\,2<q<4$. Given $\mu>0$, let $\mathcal{J}(\mu):=\inf_{u\in\Hmu(\R)}J(u)$. If 
	\begin{equation}
	\label{J neg}
	-\infty<\mathcal{J}(\mu)<0
	\end{equation}
	then there exists a ground state of $J$ at mass $\mu$, i.e. $u\in\Hmu(\R)$ such that $J(u)=\mathcal{J}(\mu)$.
\end{proposition}

\begin{proof}
	Fix $\mu>0$ and let $\{u_n\}\subset\Hmu(\R)$ be a minimizing sequence for $J$. By Remark \ref{rem_symmetry}, $u_n$ can be taken non-negative, even and non--increasing on $(0,+\infty)$, for every $n$. By \eqref{J neg} and Remark \ref{rem_coercivity}, $\{u_n\}$ is bounded in $H^1(\R)$, so that (up to subsequences) $u_n\deb u$ in $H^1(\R)$, for some $u\in H^1(\R)$. Thus, by \cite[Proposition 1.7.1]{cazenave}, $u_n\to u$ strongly in $L^r(\R)$, for every $r\in(2,\infty]$. Then, by weak lower semicontinuity,
	\begin{equation}
	\label{wlsc_prop2}
	J(u)\leq\liminf_{n\to+\infty}J(u_n)=\mathcal{J}(\mu)\,.
	\end{equation}
	Suppose now $u\equiv0$. By \eqref{wlsc_prop2}, it then follows that $\mathcal{J}(\mu)\geq0$, contradicting \eqref{J neg}. Hence, $u\not\equiv0$ on $\R$.
	
	Moreover, if we assume $0<\uLtwo^2<\mu$, then there exists $\beta>1$ such that $\|\beta u\|_2^2=\mu$, so that $\beta u\in\Hmu(\R)$, and consequently, making use of the fact that $2<q<4$ when $J=D_q$ and $2<p<6,\,2<q<4$ when $J=F_{p,q}$,
	\[
	\mathcal{J}(\mu)\leq J(\beta u)<\beta^2J(u)<J(u)\leq\mathcal{J}(\mu),
	\]
	i.e. a contradiction again. Therefore $\uLtwo^2=\mu$, so that $u\in\Hmu(\R)$ and by \eqref{wlsc_prop2} $u$ is a ground state of $J$ at mass $\mu$.
\end{proof}

\section{Purely pointwise nonlinearity: the energy functional $D_q$}
\label{sec:D}

This section is devoted to the analysis of the energy functional involving only the pointwise nonlinearity $D_q$ as in \eqref{D_intro}. We address here both the subcritical regime $p\in(2,4)$ and the critical case $p=4$, proving Theorems \ref{THM D 1}--\ref{THM D 2}.

Recall that equation \eqref{Lstat delta} is equivalent to the equation
\begin{equation}
	\label{LSeq}
	u''=\omega u \quad\quad\forall x\neq0
\end{equation}
coupled with the matching condition \eqref{matching} at the origin.

\begin{proof}[Proof of Theorem \ref{THM D 1}]
	
	Let $\mu>0$ be fixed. We split the proof into two parts. 
	
	\medskip
	\noindent\textit{Part 1. Existence.} By Remark \ref{rem_coercivity}, $D_q$ is coercive for every $\mu>0$ and $2<q<4$, and therefore bounded from below in $\Hmu(\R)$.
	
	Moreover, given $u\in\Hmu(\R)$, the mass preserving transformation
	\begin{equation}
	\label{mp_trans}
	u_\lambda(x):=\sqrt{\lambda}u(\lambda x)\,,\qquad\lambda>0\,,
	\end{equation}
	 gives a family $\{u_\lambda\}\in\Hmu(\R)$ such that
	\[
	\DD_q(\mu)\leq D_q(u_\lambda)=\f{\lambda^2}2\udot^2-\f{\lambda^{\f q2}}q |u(0)|^q<0
	\]
	provided $\lambda$ is small enough. This proves \eqref{eq_THM D 1} and, by Proposition \ref{prop_compactness}, the proof is complete.
	
	\medskip
	\noindent\textit{Part 2. Uniqueness.} Recall that a ground states $u\in\Hmu(\R)$ of $D_q$ at mass $\mu$ is an even, non--increasing on $(0,+\infty)$, positive solution of the Cauchy problem \eqref{LSeq}. The only $H^1-$solutions to \eqref{LSeq} are
	\begin{equation}
	\label{u_sol_D}
	u(x)=(2\sqrt{\omega})^{\f1{q-2}}e^{-\sqrt{\omega}|x|}\,,\qquad x\in\R\,.
	\end{equation}
	Furthermore,
	\[
	\mu=(2\sqrt{\omega})^{\f2{q-2}}\int_\R e^{-2\sqrt{\omega}|x|}\,dx=2^{\f2{q-2}}\omega^{\f{4-q}{2(q-2)}}\,,
	\]
	
	\noindent thus showing for any given $\mu>0$ the existence of a unique $\omega$ for which \eqref{u_sol_D} is a solution of \eqref{LSeq} at mass $\mu$. Hence, the positive ground state of $D_q$ at mass $\mu$ is unique and it is given by
	\[
	u(x)=\left(\f\mu2\right)^{\f{q-2}{4-q}}\exp\left(-2^{-\f2{q-2}}\mu^{\f{q-2}{4-q}}|x|\right)\,,\qquad x\in\R\,.
	\]
\end{proof}

\begin{proof}[Proof of Theorem \ref{THM D 2}]
	
	Since, given $u\in\Hmu(\R)$,
	\[
	D_4(u)=\f12\udot^2-\f14|u(0)|^4\,,
	\]
	inequality \eqref{GNinf} implies
	\[
	D_4(u)\geq \f12\udot^2-\f14\mu\udot^2=\f12\udot^2\left(1-\f\mu2\right)\,,
	\]
	so that, if $\mu\leq2$, then 
	\[
	E(u)\geq0
	\]
	for every $u\in\Hmu(\R)$, the inequality becoming strict when $\mu<2$.
	
	Furthermore, considering $u_\lambda$ as in \eqref{mp_trans}, we get
	\[
	\DD_4(\mu)\leq\f{\lambda^2}2\udot^2-\f{\lambda^2}4|u(0)|^4=\lambda^2\left(\f12\udot^2-\f14|u(0)|^4\right)\to0
	\]
	as $\lambda\to0$, we conclude that $\DD_4(\mu)=0$ for every $\mu\leq2$.
	
	If $\mu<2$, then the fact that $E(u)>0$ for every $u\in\Hmu(\R)$ ensures that ground states at mass $\mu$ do not exist.
	
	Consider then $\mu=2$ and suppose that there exists a ground state  $u\in H_2^1(\R)$ of $D_4$ which, by Remark \ref{rem_symmetry}, can be taken non--negative, even and non--increasing on $(0,+\infty)$, so that $|u(0)|=\uinfty$. Then, $D_4(u)=\DD_4(2)=0$ leads to
	\[
	\uinfty^4=2\udot^2\,,
	\]
	that is, $u$ achieves equality in Gagliardo--Nirenberg inequality \eqref{GNinf}. By uniqueness of the optimizers of \eqref{GNinf} (see Section \ref{sec:prelim}), it follows that there exists a unique family $\{\chi_\lambda\}_{\lambda>0}\subset H_2^1(\R)$ of ground states of $D_4$ at mass $\mu=2$, given by
	\[
	\chi_\lambda(x):=\sqrt{2\lambda}e^{-\lambda|x|}\,,\qquad x\in\R\,.
	\]
	Finally, fix $\mu>2$ and let $u:=\sqrt{\f\mu2}\chi_1$, so that $u\in\Hmu(\R)$ and $u$ realises equality in \eqref{GNinf}
	\[
	\uinfty^2=\mu^{1/2}\udot\,.
	\]
	Thus
	\[
	D_4(u)=\f12\udot^2-\f14|u(0)|^4=\f12\udot^2\left(1-\f\mu2\right)<0\,,
	\]
	and letting $u_\lambda$ be as in \eqref{mp_trans} 
	\[
	\DD_4(\mu)\leq D_4(u_\lambda)=\lambda^2 D_4(u)\to-\infty\qquad\text{as }\lambda\to+\infty
	\]
	completing the proof of \eqref{eq_THM D 2}.
\end{proof}

\section{The energy functional $F_{p,q}$: the subcritical regime $2<p<6,\,2<q<4$}

In this section we begin the investigation of the model involving both the standard and the pointwise nonlinearity. 

Our aim here is to prove Theorem \ref{THM F 1}, focusing on the regime $2<p<6,\,2<q<4$ where both the nonlinearities are $L^2$--subcritical. 
First, we state the following preliminary result. 

\begin{prop}
	\label{prop_uniq}
	Let $p\in(2,6]$, $q\in(2,4]$. For every $\omega>0$, there exists a unique positive solution $\varphi_\omega\in H^1(\R)$ of \eqref{NLSE}--\eqref{matching}.
\end{prop}

\begin{proof}
	
	If $u\in H^1(\R)$ is a solution of \eqref{NLSE}--\eqref{matching}, then by uniqueness of the solution of $u''+|u|^{p-2}u=\omega u$ on $\R$, it follows that $u$ is the restriction of suitable translations of the soliton $\phi_\omega$ both on $(-\infty,0)$ and on $(0,+\infty)$, i.e.
	\[
	u(x)=\begin{cases}
	\phi_\omega(x-a_-) & \text{if }x<0\\
	\phi_\omega(x-a_+) & \text{if }x\geq0
	\end{cases}
	\]
	for some $a_-,a_+\in\R$ to be determined.
	
	As a useful notation, we rewrite \eqref{soliton} in the form
	\begin{equation}
		\label{sol_s}
		\phi_\omega(x)=\left[(\sigma+1)\omega\left(1-\tanh^2\left(\sigma\sqrt{\omega}x\right)\right)\right]^{\f 1{2\sigma}}
	\end{equation}
	with $\sigma:=\f p2-1$, so that $\sigma\in(0,2)$.
	
	On the one hand, imposing $u$ to be continuous at $x=0$, we get $|a_-|=|a_+|$, so that we can write $a_\pm=\varepsilon_\pm a$, where $a>0$ and $\varepsilon_\pm:=\sgn(a_\pm)$.
	
	On the other hand, requiring $u$ to fulfil $u'(0^-)-u'(0^+)=u(0)|u(0)|^{q-2}$ we have, due to \eqref{sol_s},
	\begin{equation}
	\label{dercond}
	\sqrt{\omega}\tanh(\sigma\sqrt{\omega}a)(\epsilon_+-\epsilon_-)=-\phi_\omega(a)^{q-2},
	\end{equation}
	from which we deduce that $\epsilon_->\epsilon_+$. This forces $\epsilon_-=+1,\,\epsilon_+=-1$ and consequently 
	\[
	u(x)=
	\begin{cases}
	\phi_\omega(x-a)\quad x<0,\\
	\phi_\omega(x+a)\quad x\ge0\,.
	\end{cases}
	\]	
	Now, relying again on the explicit formula \eqref{sol_s}, relation \eqref{dercond} can be rewritten as
	\begin{equation}
	\label{tanhrel}
	\frac{\tanh(\sigma\sqrt{\omega}a)}{(1- \tanh^2(\sigma\sqrt{\omega}a))^{\frac{q-2}{2\sigma}}}=\frac{(\sigma+1)^{\frac{q-2}{2\sigma}}\omega^{\frac{q-2-\sigma}{2\sigma}}}{2}
	\end{equation}
	that is, setting $\tbar:=\tanh(\sigma\sqrt{\omega}a)$ and 
	\[
	f(\tbar):=\frac{\tbar}{(1-\tbar^2)^{\frac{q-2}{2\sigma}}}\,,
	\]
	we get
	\begin{equation}
	\label{tbarrel}
	f(\tbar)=\frac{(\sigma+1)^{\frac{q-2}{2\sigma}}\omega^{\frac{q-2-\sigma}{2\sigma}}}{2}.
	\end{equation} 
	Observing that $f((0,1))=(0,+\infty)$ and
	\begin{equation*}
	f'(t)=\frac{\frac{q-2-\sigma}{\sigma}t^2+1}{(1-t^2)^\frac{q-2+2\sigma}{2\sigma}}>0
	\end{equation*}
	for every $t\in(0,1)$, it follows that, for every frequency $\omega>0$, there exists a unique solution $\tbar$ of \eqref{tbarrel}. Since the correspondence between $\tbar$ and $a$ is one--to--one, we conclude.
\end{proof}

	We can now prove our main result in the case of both nonlinearities being subcritical.

\label{sec:subcrit}

\begin{proof}[Proof of Theorem \ref{THM F 1}]
	
	Fix $\mu>0$, $p\in(2,6)$ and $q\in(2,4)$. We divide the proof in two steps.
	
	\medskip
	\noindent\textit{Part 1. Existence.} Coercivity of $F_{p,q}$ in $\Hmu(\R)$ is guaranteed by Remark \ref{rem_coercivity}, so that $F_{p,q}$ is lower bounded in the mass constrained space. Furthermore, taking $u\in\Hmu(\R)$, $\lambda>0$ and letting $u_\lambda$ be as \eqref{mp_trans}, we get $u_\lambda\in\Hmu(\R)$ and
	\[
	\F_{p,q}(\mu)\leq F_{p,q}(u_\lambda)=\f{\lambda^2}2\udot^2-\f{\lambda^{\f p2-1}}p\uLp^p-\f{\lambda^{\f q2}}q|u(0)|^q<0\qquad \text{for }\lambda\to0\,.
	\]
	Established the negativity of $\F_{p,q}(\mu)$ as in \eqref{eq_THM F 1}, Proposition \ref{prop_compactness} ensures that ground states at mass $\mu$ exist.
	
	\medskip
	\noindent\textit{Uniqueness.} Let $u\in\Hmu(\R)$ be a ground state at mass $\mu$. Then, $u$ is a solution of \eqref{NLSE}--\eqref{matching} for a certain value of $\omega>0$, so that, by Proposition \ref{prop_uniq}, it corresponds to the unique solution $\varphi_\omega$ of \eqref{NLSE}--\eqref{matching}.
	
	Computing the mass of $\varphi_\omega$ and imposing $\varphi_\omega\in\Hmu(\R)$, we get
	\begin{equation}
	\label{massconstraint}
	\mu=2\frac{(\sigma+1)^\frac{1}{\sigma}\omega^{\frac{1}{\sigma}-\frac{1}{2}}}{\sigma}\int_{\tbar}^1(1-s^2)^{\frac{1}{\sigma}-1}\,ds
	\end{equation}
	where $\tbar=\tbar(\omega)$ is the unique solution of \eqref{tbarrel} and as usual $\sigma=\f p2-1$.
	
	Differentiating \eqref{massconstraint} with respect to $\omega$ yields at
	\[
	\f{d\mu}{d\omega}=\f{2(\sigma+1)^{\f1\sigma}}\sigma\left[\f{2-\sigma}{2\sigma}\omega^{\f1\sigma-\f32}\int_{\tbar}^1(1-s^2)^{\f1\sigma-1}\,ds-\omega^{\f1\sigma-\f12}(1-\tbar^2)^{\f1\sigma-1}\tbar'(\omega)\right]\,.
	\]
	Being $\tbar(\omega)$ the unique solution of \eqref{tbarrel}, by the Implicit Function Theorem it follows
	\[
	\tbar'(\omega)=\f{(\sigma+1)^{\f{q-2}{2\sigma}}}2\left(\f{q-2-\sigma}{2\sigma}\right)\omega^{\f{q-2-\sigma}{2\sigma}-1}\f{(1-\tbar^2)^{1+\f{q-2}{2\sigma}}}{\tbar^2\left(\f{q-2-\sigma}{\sigma}\right)+1}\,,
	\]
	leading to
	\[
	\begin{split}
	\f{d\mu}{d\omega}=\f{2(\sigma+1)^{\f1\sigma}}\sigma\omega^{\f1\sigma-\f32}&\left[\f{2-\sigma}{2\sigma}\int_{\tbar}^1(1-s^2)^{\f1\sigma-1}\,ds\right.\\
	-&\left.\f{(\sigma+1)^{\f{q-2}{2\sigma}}}2\left(\f{q-2-\sigma}{2\sigma}\right)\omega^{\f{q-2-\sigma}{2\sigma}}\f{(1-\tbar^2)^{\f q{2\sigma}}}{\tbar^2\left(\f{q-2-\sigma}{\sigma}\right)+1}\right]
	\end{split}
	\]
	and recalling \eqref{tanhrel} the second term in the square bracket can be further simplified to
	\begin{equation}
		\label{dermu}
		\f{d\mu}{d\omega}=\f{2(\sigma+1)^{\f1\sigma}}\sigma\omega^{\f1\sigma-\f32}\left[\f{2-\sigma}{2\sigma}\int_{\tbar}^1(1-s^2)^{\f1\sigma-1}\,ds-\f{\left(\f{q-2-\sigma}{2\sigma}\right)\tbar(1-\tbar^2)^{\f1\sigma}}{\tbar^2\left(\f{q-2-\sigma}{\sigma}\right)+1}\right]\,.
	\end{equation}
	Taking advantage of formula \eqref{dermu}, we now show that
	\begin{equation}
	\label{derneg}
	\f{d\mu}{d\omega}>0
	\end{equation}
	for every $\omega>0$, $\sigma\in(0,2)$ and $q\in(2,4)$. 
	
	If $q-2-\sigma\leq0$, then \eqref{derneg} is immediate, since $\tbar^2\left(\f{q-2-\sigma}\sigma\right)+1>0$ for every $\tbar\in(0,1)$ and the square bracket is the sum of two positive terms.
	
	Consider thus $q-2-\sigma>0$ and set for every $\tbar\in(0,1)$
	\[
	F(\tbar):=\f{2-\sigma}{2\sigma}\int_{\tbar}^1(1-s^2)^{\f1\sigma-1}\,ds-\f{\left(\f{q-2-\sigma}{2\sigma}\right)\tbar(1-\tbar^2)^{\f1\sigma}}{\tbar^2\left(\f{q-2-\sigma}{\sigma}\right)+1}\,.
	\]
	On the one hand, we have
	\[
	\lim_{\tbar\to0^+}F(\tbar)=\f{2-\sigma}{2\sigma}\int_0^1(1-s^2)^{\f1\sigma-1}\,ds>0\,,\qquad\lim_{\tbar\to1^-}F(\tbar)=0\,.
	\]
	On the other hand, differentiating $F$ with respect to $\tbar$ gives
	\[
	\begin{split}
	F'(\tbar)=-&\f{2-\sigma}{2\sigma}(1-\tbar^2)^{\f1\sigma-1}\\
	-&\f{(1-\tbar^2)^{\f1\sigma-1}}{2\sigma\left(\f{\tbar^2}\sigma+\f1{q-2-\sigma}\right)^2}\left[\f1{q-2-\sigma}-\f q{\sigma(q-2-\sigma)}\tbar^2-\f{2-\sigma}{\sigma^2}\tbar^4\right]\,,
	\end{split}
	\]
	and direct computations show that $F'(\tbar)<0$ for every $\tbar\in(0,1)$. Hence, $F(\tbar)$ is strictly positive in $(0,1)$ and so does $\f{d\mu}{d\omega}$, thus proving \eqref{derneg}.
	
	Finally, \eqref{derneg} implies that the mass $\mu$ of the ground state of $F_{p,q}$ in $\Hmu(\R)$ is a strictly increasing function of $\omega$. Therefore, for every $\mu>0$ there exists a unique $\omega>0$ such that $\varphi_\omega$ as in Proposition \ref{prop_uniq} is the required ground state.
\end{proof}

\section{The energy functional $F_{p,q}:$ the cases $p=6,\,2<q<4$ and $2<p<6,\,q=4$}
\label{sec:single_crit}

Throughout this section, we discuss the behaviour of the minimization problem \eqref{infF_intro} when one of the two nonlinearities is subcritical and the other is critical. Here follows the proof of Theorem \ref{THM F 2}, dealing with the cases $p=6,\,2<q<4$ and $2<p<6,\,q=4$.

\begin{proof}[Proof of Theorem \ref{THM F 2}]
	
	We begin by proving statement {\em(i)}. Given $q\in(2,4)$ and $\mu>0$, plugging \eqref{GN6} and \eqref{GNinf} into $F_{6,q}$ yields
	\[
	F_{6,q}(u)\geq\f12\udot^2-\f{K_6}6\mu^2\udot^2-\f1q\mu^{\f q4}\udot^{\f q2}=\f12\left(1-\f{K_6\mu^2}3\right)\udot^2-\f1q\mu^{\f q4}\udot^{\f q2}
	\]
	for every $u\in\Hmu(\R)$. Recalling the actual value of $K_6=\f4{\pi^2}$, we get that, if $\mu<\f{\sqrt{3}}2\pi$, the coefficient of $\udot^2$ in the right--hand side above is positive and, since $q\in(2,4)$, $F_{6,q}$ is coercive in $\Hmu(\R)$. Hence, 
	\[
	\F_{6,q}(\mu)>-\infty\,.
	\]
	Moreover, given $u\in\Hmu(\R)$ and taking $u_\lambda$ as in \eqref{mp_trans}, then $u_\lambda\in\Hmu(\R)$ for every $\lambda>0$ and
	\[
	\F_{6,q}(\mu)\leq F_{6,q}(u_\lambda)=\f{\lambda^2}2\udot^2-\f{\lambda^{2}}6\uLsix^6-\f{\lambda^{\f q2}}q|u(0)|^q<0\qquad\text{for }\lambda\to0\,.
	\]
	This proves the first part of \eqref{eq_THM F 2} and implies existence of ground states for every $\mu<\f{\sqrt{3}}2\pi$ due to Proposition \ref{prop_compactness}. Moreover, uniqueness of these ground states can be proved repeating the argument in the second part of the proof of Theorem \ref{THM F 1}, that works with no changes in the case $p=6,\,2<q<4$, too.
	
	Conversely, given $\mu\geq\f{\sqrt{3}}2\pi$, let $\{\phi_\lambda\}_{\lambda>0}\in H_{\f{\sqrt{3}}2\pi}^1(\R)$ be the family of critical solitons attaining equality in Gagliardo--Nirenberg inequality \eqref{GN6} at mass $\f{\sqrt{3}}2\pi$, so that
	\begin{equation}
		\label{opt_GN6}
		\|\phi_\lambda\|_6^6=\f4{\pi^2}\left(\f{\sqrt{3}}2\pi\right)^2\|\phi_\lambda'\|_2^2=3\|\phi_\lambda'\|_2^2\,.
	\end{equation}
	Setting $u_{\mu,\lambda}:=\sqrt{\f\mu{\f{\sqrt{3}}2\pi}}\phi_\lambda$, then, $u_{\mu,\lambda}\in\Hmu(\R)$ and using \eqref{opt_GN6}
	\[
	\begin{split}
	F_{6,q}(u_{\mu,\lambda})=&\f12\f\mu{\f{\sqrt{3}}2\pi}\|\phi_\lambda'\|_2^2-\f16\left(\f\mu{\f{\sqrt{3}}2\pi}\right)^3\|\phi_\lambda\|_6^6-\f1q\left(\f\mu{\f{\sqrt{3}}2\pi}\right)^{\f q2}|\phi_\lambda(0)|^q\\
	=&\f{\lambda^2}2\f\mu{\f{\sqrt{3}}2\pi}\|\phi_1'\|_2^2\left(1-\left(\f\mu{\f{\sqrt{3}}2\pi}\right)^2\right)-\f{\lambda^{\f q2}}q\left(\f\mu{\f{\sqrt{3}}2\pi}\right)^{\f q2}|\phi_\lambda(0)|^q\to-\infty
	\end{split}
	\]
	
	\noindent as $\lambda\to+\infty$, thus concluding the proof of \eqref{eq_THM F 2}.

	The proof of statement {\em(ii)} is analogous to the previous case.
	Indeed, let $\mu>0$, $p\in(2,6)$ and $q=4$. Plugging \eqref{GNp} and \eqref{GNinf} into $F_{p,4}$, we have, for every $u\in\Hmu(\R)$
	\[
	F_{p,4}(u)\geq\f12\left(1-\f\mu2\right)\udot^2-\f{K_p \mu^{\f p4+\f 12}}p\udot^{\f p4-\f12}
	\]
	
	\noindent so that, if $\mu<2$, then $F_{p,4}(u)\to+\infty$ as $\udot\to+\infty$ and the energy is coercive and lower bounded in $\Hmu(\R)$. Therefore, arguing as above, $\F_{p,4}(\mu)<0$ and ground states exist and are unique for every $\mu<2$.
	
	On the contrary, in the case $\mu\geq2$, let $\chi_\lambda(x)=\sqrt{2\lambda}e^{-\lambda|x|}$ be optimal in the Gagliardo--Nirenberg inequality \eqref{GNinf}, i.e.
	\[
	|\chi_\lambda(0)|^2=\|\chi_\lambda\|_\infty^2=\|\chi_\lambda\|_2\|\chi_\lambda'\|_2=\sqrt{2}\|\chi_\lambda'\|_2
	\]
	
	\noindent and set $u_{\mu,\lambda}:=\sqrt{\f\mu2}\chi_\lambda$. Thus, $u_{\mu,\lambda}\in\Hmu(\R)$ and it follows
	\[
	\begin{split}
	\F_{p,4}(\mu)&\leq F_{p,4}(u_{\mu,\lambda})=\f\mu4\|\chi_\lambda'\|_2^2-\f1p\left(\f\mu2\right)^{\f p2}\|\chi_\lambda\|_p^p-\f14\left(\f\mu2\right)^2|\chi_\lambda(0)|^4\\
	&=\f{\lambda^2\mu}4\|\chi'\|_2^2\left(1-\f\mu2\right)-\f{\lambda^{\f p2-1}}p\left(\f\mu2\right)^{\f p2}\|\chi\|_p^p\to-\infty
	\end{split}
	\]
	
	\noindent when $\lambda\to+\infty$, and \eqref{eq_THM F 3} is proved.
\end{proof}

\section{The energy functional $F_{6,4}$}
\label{sec:F64}

Here we give the proof of Theorem \ref{THM F 4}, dealing with the case of both nonlinearities at their corresponding critical powers. We preliminary notice that, defined $u_\lambda(x)=\sqrt{\lambda}u(\lambda x)$, one has
\begin{equation}
\label{l^2 F}
F_{6,4}(u_\lambda)=\lambda^2F_{6,4}(u)\,.
\end{equation}
As a consequence, letting $\lambda\to0$ in \eqref{l^2 F}, it follows $\F_{6,4}(\mu)\leq0$ for every $\mu>0$. Moreover, if there exists $v\in\Hmu(\R)$ such that $F_{6,4}(v)<0$, then as $\lambda\to+\infty$ we get $\F_{6,4}(\mu)=-\infty$.

\begin{lem}
	\label{lem1}
	Let $\mu_2>\mu_1>0$. 
	\begin{itemize}
		\item[(i)] if $\F_{6,4}(\mu_1)=-\infty$, then $\F_{6,4}(\mu_2)=-\infty$;
		\item[(ii)] if $\F_{6,4}(\mu_2)=0$, then $\F_{6,4}(\mu_1)=0$.
	\end{itemize}

	\begin{proof}
		Assume first that $\F_{6,4}(\mu_1)=-\infty$. Therefore, there exists $u\in H_{\mu_1}^1(\R)$ such that $F_{4,6}(u)<0$. Then, setting $v:=\sqrt{\f{\mu_2}{\mu_1}}u$, we have $v\in H_{\mu_2}^1(\R)$ and
		\[
		\begin{split}
		 F_{6,4}(v)=&\f12\f{\mu_2}{\mu_1}\udot^2-\f16\left(\f{\mu_2}{\mu_1}\right)^3\uLsix^6-\f14\left(\f{\mu_2}{\mu_1}\right)^2|u(0)|^4\\
		<&\f{\mu_2}{\mu_1}F_{6,4}(u)<0\,,
		\end{split}
		\]
		making use of the fact that $\mu_2/\mu_1>1$. We conclude that $\F_{6,4}(\mu_2)=-\infty$.
		
		The proof of statement \textit{(ii)} is analogue to the proof of {\em(i)}. Indeed, assuming by contradiction that $\F_{6,4}(\mu_1)\neq0$, then $\F_{6,4}(\mu_1)<0$. Hence, there exists $u\in H_{\mu_1}^1(\R)$ realizing strictly negative energy $F_{6,4}(u)<0$ and repeating the previous argument yields $\F_{6,4}(\mu_2)=-\infty$, a contradiction.
	\end{proof}
\end{lem}

Let us introduce
\begin{equation}
	\label{mu*}
	\mu^*:=\sup\{\,\mu\geq0\,:\,\F_{6,4}(\mu)=0\}\,.
\end{equation}

\begin{remark}
	\label{rem mu*}
	Suppose $\mu^*>0$. Then, by Lemma \ref{lem1} and definition of $\mu^*$, it follows 
	\begin{equation}
	\label{0 infty mu*}
	\begin{split}
	\F_{6,4}(\mu)=&\,0 \qquad\qquad\text{if }\mu<\mu^*\\
	\F_{6,4}(\mu)=&-\infty \,\,\,\qquad\text{if }\mu>\mu^*\,.
	\end{split}
	\end{equation}
	Furthermore, given $\mu>0$ and rewriting every $v\in\Hmu(\R)$ as $v=\sqrt{\mu}u$, for a suitable $u\in H_1^1(\R)$, it is readily seen that
	\[
	\F_{6,4}(\mu)=\inf_{u\in H_1^1(\R)}f_u(\mu)
	\]
	where, for every $u\in H_1^1(\R)$, we set
	\[
	f_u(\mu):=F_{6,4}(\sqrt{\mu}u)=\f\mu2\udot^2-\f{\mu^3}6\uLsix^6-\f{\mu^2}4|u(0)|^4\,.
	\]
	As $f_u(\mu)$ is a continuous function of $\mu$ for every fixed $u\in H_1^1(\R)$, $\F_{6,4}(\mu)$ is an upper semicontinuous function of the mass. By \eqref{0 infty mu*}, this entails
	\[
	\F_{6,4}(\mu^*)\geq\limsup_{n\to+\infty}\F_{6,4}(\mu_n)=0
	\]
	for every sequence of masses $\{\mu_n\}$ such that $\mu_n<\mu^*$ for all $n$, $\mu_n\to\mu^*$ as $n\to+\infty$. Since $\F_{6,4}(\mu)\leq0$ for every $\mu$, we then have $\F_{6,4}(\mu^*)=0$. 
\end{remark}
The next lemma guarantees that if $\mu^*$ is not equal to zero, then global minimizers of $F_{6,4}$ must exist at mass $\mu^*$.

\begin{lem}
	\label{lem2}
	If $\mu^*>0$, then ground states at mass $\mu^*$ exist, i.e. there exists $u\in H_{\mu^*}^1(\R)$ such that $F_{6,4}(u)=\F_{6,4}(\mu^*)$.
\end{lem}

\begin{proof}
	By Remark \ref{rem mu*}, $\F_{6,4}(\mu^*)=0$ and $\F_{6,4}(\mu)=-\infty$ for every $\mu>\mu^*$. Therefore, given $\mu>\mu^*$, the continuity of $F_{6,4}$ and the connection of $\Hmu(\R)$, there exists $u_\mu\in\Hmu(\R)$ such that $F_{6,4}(u_\mu)=0$. Moreover, up to a mass preserving transformation and without loss of generality, we can further assume $\|u_\mu'\|_2=1$ and $u_\mu$ even and non--increasing on $(0,+\infty)$. Hence, $\{u_\mu\}$ is bounded in $H^1(\R)$ and (up to subsequences) $u_\mu\deb u$ in $H^1(\R)$ and $u_\mu\to u$ in $L_{loc}^\infty(\R)$, for some $u\in H^1(\R)$. By \cite[Proposition 1.7.1]{cazenave} $u_\mu\to u$ strongly $L^p(\R)$ for every $p>2$ so that, coupled with weak lower semicontinuity, 
	
	\begin{equation}
	\label{eq_lem2}
	F_{6,4}(u)\leq\liminf_{\mu\to\mu^*} F_{6,4}(u_\mu)=0\,.
	\end{equation}
	Suppose now $u\equiv0$ on $\R$. Then, $F_{6,4}(u_\mu)=0$, with $u_\mu$ even, non--increasing on $(0,+\infty)$ and $u_\mu\to0$ in $L^\infty(\R)$ imply 
	\[
	\f12\|u_\mu'\|_2^2=\f16\|u_\mu\|_6^6+\f14|u_\mu(0)|^4\to0\qquad\text{as }\mu\to\mu^*\,,
	\]
	contradicting $\|u_\mu'\|_2=1$. Thus, $u\not\equiv0$.
	
	Finally, let $m:=\uLtwo^2$ and suppose $0<m<\mu^*$. Being $\F_{6,4}(m)=0$ and by \eqref{eq_lem2}, $u$ is a ground state at mass $m<\mu^*$ and $F_{6,4}(u)=0$. Therefore, setting $v=\sqrt{\f{\mu^*}{m}}u$, we get $v\in H_{\mu^*}^1(\R)$ realizing
	\[
	F(v)=F\left(\sqrt{\f{\mu^*}{m}}u\right)<\f{\mu^*}m F(u)=0
	\]
	since $\mu^*/m>1$. This is impossible, since in the first part of the proof we already showed that $\F_{6,4}(\mu^*)=0$. Thus, $u\in H_{\mu^*}^1(\R)$ and $F_{6,4}(u)=0$ by \eqref{eq_lem2}, that is $u$ is a ground state at mass $\mu^*$.
\end{proof}

We can now prove Theorem \ref{THM F 4}.

\begin{proof}[Proof of Theorem \ref{THM F 4}]
	On the one hand, plugging \eqref{GN6}--\eqref{GNinf} into the energy, we have
	\[
	F_{6,4}(u)\geq\f12\udot^2-\f4{6\pi^2}\mu^2\udot^2-\f14\mu\udot^2=\f12\udot^2\left(1-\f4{3\pi^2}\mu^2-\f\mu2\right)
	\]
	thus showing that $F_{6,4}(u)>0$ for every $u\in\Hmu(\R)$, provided $\mu$ is small enough. Hence, $\mu^*>0$ and, by Remark \ref{rem mu*} we get \eqref{eq_THM F 4}.
	
	On the other hand, ground states of $F_{6,4}$ in $\Hmu(\R)$ are solutions of \eqref{NLSE}--\eqref{matching} for some Lagrange multiplier $\omega>0$, so that, by Proposition \ref{prop_uniq}, they must correspond to certain $\varphi_\omega\in H^1(\R)$. Moreover, if $p=6$ and $q=4$, equation \eqref{tbarrel} simply becomes
	\[
	\f\tbar{\sqrt{1-\tbar^2}}=\f{\sqrt{3}}2\,,
	\] 
	
	\noindent that is $\tbar=\sqrt{\f37}$. Hence, computing explicitly the mass of $\varphi_\omega$ gives
	\[
	\mu=\sqrt{3}\int_{\sqrt{\f37}}^1(1-s^2)^{-\f12}\,ds=\sqrt{3}\left(\f\pi2-\text{arcsin}\left(\sqrt{\f37}\right)\right)\,,
	\] 
	
	\noindent which implies that, regardless of $\omega$, all solutions of \eqref{NLSE}--\eqref{matching} share the same value of the mass. Since, by Lemma \ref{lem2}, ground states at mass $\mu^*$ must exist, we conclude that 
	\[
	\mu^*=\sqrt{3}\left(\f\pi2-\text{arcsin}\left(\sqrt{\f37}\right)\right)
	\]
	
	\noindent and ground states exist if and only if $\mu=\mu^*$. In fact, a direct computation shows that $F_{6,4}(\varphi_\omega)=0$ for every $\omega>0$, so that all solutions $\{\varphi_\omega\}_{\omega>0}$ of \eqref{NLSE}--\eqref{matching} are ground states at the critical mass.
\end{proof}

\end{document}